\documentclass[10pt]{article}
\usepackage{}
\usepackage{amsfonts}
\usepackage{latexsym,bm}
\usepackage{amssymb}
\usepackage{amsmath}
\usepackage{amsthm}
\usepackage{mathrsfs}
\usepackage[all]{xy}
\numberwithin{equation}{section}

\theoremstyle{plain}
\newtheorem{thm}{Theorem}[section]
\newtheorem{lem}[thm]{Lemma}
\newtheorem{prop}[thm]{Proposition}
\newtheorem{cor}[thm]{Corollary}
\theoremstyle{remark}
\newtheorem{rmk}[thm]{Remark}
\newtheorem{e.g.}[thm]{Example}
\theoremstyle{definition}
\newtheorem{defi}[thm]{Definition}

\begin{document}

\title{\Large Quadro-quadric special birational transformations from projective spaces to smooth complete intersections}
\author{\normalsize Qifeng Li}
\date{}
\maketitle

\begin{abstract}
Let $\phi: \mathbb{P}^{r}\dashrightarrow Z$ be a birational transformation with a smooth connected base locus scheme, where $Z\subseteq\mathbb{P}^{r+c}$ is a nondegenerate prime Fano manifold. We call $\phi$ a quadro-quadric special briational transformation if $\phi$ and $\phi^{-1}$ are defined by linear subsystems of $|\mathcal{O}_{\mathbb{P}^{r}}(2)|$ and $|\mathcal{O}_{Z}(2)|$ respectively.
In this paper we classify quadro-quadric special birational transformations in the cases where either $(i)$ $Z$ is a complete intersection and the base locus scheme of $\phi^{-1}$ is smooth, or $(ii)$ $Z$ is a hypersurface.
\end{abstract}




\section{\normalsize Introduction}

We work over the complex number field. Varieties are assumed to be irreducible and reduced unless otherwise stated. A smooth projective variety $V\subseteq\mathbb{P}^{N}$ is called a \textit{prime Fano manifold} if $\text{Pic}(V)=\mathbb{Z}(\mathcal{O}_{V}(1))$ and $V$ is covered by lines. Let $\phi: \mathbb{P}^{r}\dashrightarrow \mathbb{P}^{N}$ be a rational map defined by an $N$-dimensional linear subsystem of $|\mathcal{O}_{\mathbb{P}^{r}}(a)|$, and $Z$ be the closure of the image. Assume $\phi: \mathbb{P}^{r}\dashrightarrow Z$ is birational, and $Z\subseteq\mathbb{P}^{N}$ is a prime Fano manifold. Denote by $X$ (resp. $Y$) the base locus scheme of $\phi$ (resp. $\phi^{-1}$). If moreover $X$ is smooth and connected, then we call $\phi$ a \textit{special birational transformation}. Assume that  $\phi^{-1}$ is defined by a linear subsystem of $|\mathcal{O}_{Z}(b)|$. Then $\phi$ is said to be of type $(a, b)$. We call $\phi$ a \textit{quadratic} (resp. \textit{quadro-quadric}) special birational transformation if $a=2$ (resp. $a=b=2$).

It is hard to classify special birational transformations of type $(a, b)$, even if $Z$ is a projective space. Fu and Hwang classified special birational transformation of type $(2, 1)$ in \cite{FH14}. Ein and Shepherd-Barron showed in \cite{ES89} that if $\phi: \mathbb{P}^{r}\dashrightarrow \mathbb{P}^{r}$ is a quadro-quadric special birational transformation, then $X$ and $Y$ are Severi varieties.  Alzati and Sierra classified in \cite{AS13b} quadro-quadric special birational transformations to LQEL-manifolds. Staglian\`{o} studied in \cite{Sta12} quadratic special birational transformations to hypersurfaces. In particular, he described quadro-quadric special birational transformations to smooth quadric hypersurfaces.

Our work is motivated by \cite{Sta12}. The first main result in this paper is as follows:

\begin{thm} \label{introduction: thm. Z hypersurface cases}
Let $\phi: \mathbb{P}^{r}\dashrightarrow Z$ be a quadro-quadric special birational transformation. Assume that $Z\subsetneq\mathbb{P}^{r+1}$ is a nondegenerate smooth hypersurface. Denote by $X$ (resp. $Y$) the base locus scheme of $\phi$ (resp. $\phi^{-1}$). Then $Z$ is a quadric hypersurface, $Y$ is a Severi variety and $X$ is a nonsingular hyperplane section of a Severi variety.
\end{thm}

Remark that there is a classification of Severi varieties due to Zak (see Proposition \ref{classification of Severi varieties} in the following). The key observation for the proof of Theorem \ref{introduction: thm. Z hypersurface cases} is that the VMRT $Z^{(1)}$ of $Z$ is covered by lines, which was proved in \cite{AS13b} and implies that the index $i(Z^{(1)})\geq 2$. On the other hand, as $Z$ is smooth hypersurface, we have $i(Z^{(1)})=r+1-\frac{\deg(Z)(\deg(Z)+1)}{2}$. Combining with the properties of $X$, we get $\dim(X)\leq 33$, and the possible values of $\dim(X)$ and $r$. Most cases can be ruled out in a standard way with the help of the Divisibility Theorem for QEL-manifolds, except one case with $\dim(X)=25$ and $r=43$, where we need to explore some delicate structure of entry loci of QEL-manifolds.

For the complete intersection cases, we get the following

\begin{thm} \label{introduction: thm. c.i. and Y smooth implies c=1 and deg(Z)=2}
Let $\phi: \mathbb{P}^{r}\dashrightarrow Z$ be a quadro-quadric special birational transformation with base locus scheme $X$. Assume that $Z\subsetneq\mathbb{P}^{N}$ is a nondegenerate smooth complete intersection, and the base locus scheme $Y$ of $\phi^{-1}$ is smooth. Then $Z$ is a quadric hypersurface, $Y$ is a Severi variety and $X$ is a nonsingular hyperplane section of a Severi variety.
\end{thm}

The key point is that by studying the secant variety of $Y$, we can show that $Y$ is a Severi variety and that $Z$ is a complete intersection of quadric hypersurfaces. Then the theorem follows from a technique result on Severi varieties.

The paper is organized as follows. In Section \ref{section: preliminary}, we recall some facts about Severi varieties, quadratic manifolds, QEL-manifolds, and conic-connected manifolds. In Section \ref{section: type (2, b)}, we study quadratic special birational transformations to prime Fano manifolds. In Section \ref{section: proof of main theorems}, we prove Theorem \ref{introduction: thm. Z hypersurface cases} and Theorem \ref{introduction: thm. c.i. and Y smooth implies c=1 and deg(Z)=2}. In Section \ref{section: Proof of non-existence of certain 25-dim. quadratic QEL-manifolds}, we prove Proposition \ref{not exist 25-dim. quadric SQEL-mfd with delta=9}, which claims the non-existence of certain quadratic QEL-manifolds, and helps to complete our proofs in the previous sections.

\textbf{\normalsize Acknowledgements.} It is my great pleasure to thank Baohua Fu for a lot of discussions and suggestions. I also want to express the gratitude to Alberto Alzati, Giovanni Staglian\`{o} and Fedor L. Zak for communications.

\section{\normalsize Preliminaries} \label{section: preliminary}

This section is to collect some facts about Severi varieties, quadratic manifolds, QEL-manifolds, and conic-connected manifolds, which will be frequently used.

\subsection{\normalsize Severi varieties}

Let $V\subset\mathbb{P}^{r}$ be a closed subvariety. The secant variety $\text{Sec}(V)$ of $V$ is the closure of the union of the secant lines of $V$. The nonnegative integer $\delta(V):=2\,\dim(V)+1-\dim(Sec(V))$ is called the \textit{secant defect} of $V$. We have the following linear normality theorem due to Zak:

\begin{prop} (\cite[Cor. V.1.13]{Zak93}) \label{sec(V) neq P^r then delta<=n/2}
Let $V\subseteq\mathbb{P}^{r}$ be a nondegenerate smooth projective variety of dimension $n$ with secant defect $\delta$. If $\text{Sec}(V)\neq\mathbb{P}^{r}$, then $\delta\leq\frac{n}{2}$ and $n\leq\frac{2}{3}(r-2)$.
\end{prop}

Let $V\subseteq\mathbb{P}^{r}$ be a nondegenerate smooth projective variety. If $\dim(V)=\frac{2}{3}(r-2)$ and $\text{Sec}(V)\neq\mathbb{P}^{r}$, then $V$ is called a \textit{Severi variety}.
A famous result of F. L. Zak is the following classification.

\begin{prop} (\cite[Thm. IV.4.7]{Zak93}) \label{classification of Severi varieties}
A Severi variety is projectively equivalent to one of the following:

$(a)$ the Veronese surface $v_{2}(\mathbb{P}^{2})\subseteq\mathbb{P}^{5}$;

$(b)$ the Segre embedding $\mathbb{P}^{2}\times\mathbb{P}^{2}\subseteq\mathbb{P}^{8}$;

$(c)$ the Pl\"{u}cker embedding $\mathbb{G}(1, 5)\subseteq\mathbb{P}^{14}$;

$(d)$ the $E_{6}$-variety $\mathbb{OP}^{2}\subseteq\mathbb{P}^{26}$.
\end{prop}

\subsection{\normalsize Quadratic manifolds} \label{subsection: quadratic manifolds}


Let $V\subseteq\mathbb{P}^{r}$ be a smooth projective variety covered by lines. Take a point $v\in V$. Denote by $T_{v}V$ and $\mathbb{T}_{v}V$ the affine tangent space and the embedded tangent space of $V$ at $v$, respectively. Let $\mathcal{L}_{v}(V)\subseteq\mathbb{P}((T_{v}V)^{*})$ be the variety of lines in $V$ passing through $v$. When $v$ is a general point in $V$, we call $\mathcal{L}_{v}(V)$ the \textit{variety of minimal rational tangents} (VMRT for short) of $V$ at $v$, and if there is no confusion, we denote by $V^{(1)}=\mathcal{L}_{v}(V)$. Denote by $V^{(0)}=V$ and $V^{(i+1)}$ the VMRT (if exists) of $V^{(i)}$ at a general point $v_{i}\in V^{(i)}$. If $V\subseteq\mathbb{P}^{r}$ is a smooth projective variety covered by lines, then $V^{(1)}$ is a (possibly reducible) smooth projective variety (see \cite[Prop. 1.5]{Hwa00}).

A smooth projective variety $V\subseteq\mathbb{P}^{r}$ is called a \textit{quadratic manifold}, if it is the scheme-theoretic intersection of quadric hypersurfaces. If $V\subseteq\mathbb{P}^{r}$ is a quadratic manifold covered by lines, then $V^{(1)}$ is a (possibly reducible) quadratic manifold (see \cite[Prop. 2.2]{Rus12}).

\subsection{\normalsize QEL-manifolds} \label{subsection: QEL-manifolds}

Let $V\subseteq\mathbb{P}^{r}$ be a projective variety. For a point $u\in\text{Sec}(V)\backslash V$, denote by $C_{u}(V)$ the closure of the union of secant lines of $V$ passing through $u$. Let $\Sigma_{u}(V)$ be the closure of the set of points $v\in V$ such that there is a secant line of $V$ passing through $u$ and $v$. We call $C_{u}(V)$ the \textit{secant cone} of $u$ in $\text{Sec}(V)$ and call $\Sigma_{u}(V)$ the \textit{entry locus} of $u$ in $V$. When there is no confusion, denote by $C_{u}=C_{u}(V)$ and $\Sigma_{u}=\Sigma_{u}(V)$.

A smooth projective variety $V\subseteq\mathbb{P}^{r}$ is called a QEL-manifold if for a general point $u\in\text{Sec}(V)$, the entry locus $\Sigma_{u}$ is a $\delta(V)$-dimensional quadric hypersurface. When a QEL-manifold is also a quadratic manifold, we call it a quadratic QEL-manifold.

Let $V\subseteq\mathbb{P}^{r}$ be a nondegenerate QEL-manifold of dimension $n$ with secant defect $\delta$. If $\text{Sec}(V)=\mathbb{P}^{r}$, then $V$ is linearly normal. Now assume $\delta>0$. Then for a general point $u\in\text{Sec}(V)\backslash V$, $\Sigma_{u}$ is an irreducible smooth quadric hypersurface of dimension $\delta$. Through two general points in $V$ there passes a unique $\delta$-dimensional quadric hypersurface. Moreover, this quadric hypersurface is irreducible and smooth. Assume $\text{Sec}(V)\neq\mathbb{P}^{r}$ and $\delta>0$. Let $L$ be a linear subspace such that $\dim(L)=r-\dim(\text{Sec}(V))-1$ and $L\cap\text{Sec}(V)=\emptyset$. Denote by $\pi: \mathbb{P}^{r}\dashrightarrow\mathbb{P}^{2n+1-\delta}$ the linear projection from $L$. Then $\pi(V)$ is not a QEL-manifold, since for a general point $u\in\mathbb{P}^{2n+1-\delta}$, $\Sigma_{u}(\pi(V))$ has $\deg(\text{Sec}(V))$ irreducible components. For details of the discussions in this paragraph, see \cite[page 600-601]{Rus09}.

The following Proposition is from \cite[Thm. 2.1, Thm. 2.3, Thm. 2.8]{Rus09} and \cite[Thm. 3]{Fu08}. The assertion $(ii)$ is known as the Divisibility Theorem for QEL-manifolds.

\begin{prop} \label{L^1(V) is QEL and Divisibility Theorem}
Let $V\subseteq\mathbb{P}^{r}$ be a QEL-manifold of dimension $n$ with secant defect $\delta>0$. Then the following hold.

$(i)$ If $\delta\geq 3$, then $V$ is a prime Fano manifold. Moreover, $V^{(1)}\subseteq\mathbb{P}^{n-1}$ is a QEL-manifold of dimension $\frac{n+\delta}{2}-2$ with secant defect $\delta-2$, and $\text{Sec}(V^{(1)})=\mathbb{P}^{n-1}$.

$(ii)$ If $\delta\geq 3$, then $2^{r_{V}}$ divides $n-\delta$, where $r_{V}=[\frac{\delta-1}{2}]$.

$(iii)$ If $3\leq\delta<n$, then $\delta\leq 2[\text{log}_{2}n]+2$.

$(iv)$ If $V$ is a prime Fano manifold, then the index $i(V)=\frac{n+\delta}{2}$.
\end{prop}

The following two Propositions are due to Russo (see \cite[Cor. 3.1, Cor. 3.2]{Rus09}), while the statements are slightly different from that in \cite{Rus09}. For statements here, see  \cite[Prop. 4.7, Prop. 4.8, Remark 4.9]{FH14}.

\begin{prop} \label{classification of QEL with delta>n/2}
Let $V\subseteq\mathbb{P}^{r}$ be a nondegenerate QEL-manifold of dimension $n$ with secant defect $\delta>\frac{n}{2}$ and $\delta<n$. Then $V\subseteq\mathbb{P}^{r}$ is projectively equivalent to one of the following:

$(a)$ the Segre threefold $\mathbb{P}^{1}\times\mathbb{P}^{2}\subseteq\mathbb{P}^{5}$;

$(b)$ the Pl\"{u}ker embedding $\mathbb{G}(1, 4)\subseteq\mathbb{P}^{9}$;

$(c)$ the 10-dimensional Spinor variety $S^{10}\subseteq\mathbb{P}^{15}$;

$(d)$ a nonsingular hyperplane section of $(b)$ or $(c)$.
\end{prop}

\begin{prop} \label{classification of QEL with delta=n/2}
Let $V\subseteq\mathbb{P}^{r}$ be a nondegenerate QEL-manifold of dimension $n$ with secant defect $\delta=\frac{n}{2}$. Then $V\subseteq\mathbb{P}^{r}$ is projectively equivalent to one of the following:

$(a)$ a nonsingular hyperplane section of the Segre threefold $\mathbb{P}^{1}\times\mathbb{P}^{2}\subseteq\mathbb{P}^{5}$;

$(b)$ the Veronese surface $\upsilon_{2}(\mathbb{P}^{2})\subseteq\mathbb{P}^{5}$;

$(c)$ the Segre fourfold $\mathbb{P}^{1}\times\mathbb{P}^{3}\subseteq\mathbb{P}^{7}$;

$(d)$ the Segre fourfold $\mathbb{P}^{2}\times\mathbb{P}^{2}\subseteq\mathbb{P}^{8}$;

$(e)$ a nonsingular codimension-2 linear section of $\mathbb{G}(1, 4)\subseteq\mathbb{P}^{9}$;

$(f)$ a nonsingular codimension-2 linear section of $S^{10}\subseteq\mathbb{P}^{15}$;

$(g)$ the Pl\"{u}ker embedding $\mathbb{G}(1, 5)\subseteq\mathbb{P}^{14}$;

$(h)$ the $E_{6}$-variety $\mathbb{OP}^{2}\subseteq\mathbb{P}^{26}$.
\end{prop}

\subsection{\normalsize Conic-connected manifolds}

A nondegenerate smooth projective variety $V\subseteq\mathbb{P}^{r}$ is said to be a \textit{conic-connected manifold}, if through two general points of $V$ there passes an irreducible conic contained in $V$. There is a classification of conic-connected manifolds due to Ionescu and Russo as follows:

\begin{prop} (\cite[Thm. 2.2]{IR10}) \label{classification of conic-connected manifolds}
Let $V\subseteq\mathbb{P}^{r}$ be a linearly normal nondegenerate conic-connected manifold of dimension $n$. Then either $V\subseteq\mathbb{P}^{r}$ is a prime Fano manifold of index $i(V)\geq\frac{n+1}{2}$, or it is projectively equivalent to one of the following:

$(i)$ the second Veronese embedding $\upsilon_{2}(\mathbb{P}^{n})\subseteq\mathbb{P}^{\frac{n(n+3)}{2}}$;

$(ii)$ the projection of $\upsilon_{2}(\mathbb{P}^{n})$ from the linear space $\langle\upsilon_{2}(\mathbb{P}^{m})\rangle$, where $\mathbb{P}^{m}\subseteq\mathbb{P}^{n}$ is a linear space with $0\leq m\leq n-2$ and $r=\frac{(n+1)(n+2)}{2}-\frac{(m+1)(m+2)}{2}-1$;

$(iii)$ the Segre embedding $\mathbb{P}^{m}\times\mathbb{P}^{n-m}\subseteq\mathbb{P}^{r}$, where $1\leq m\leq n-1$ and $r=m(n-m)+n$;

$(iv)$ a nonsingular hyperplane section of the Segre embedding $\mathbb{P}^{m}\times\mathbb{P}^{n+1-m}\subseteq\mathbb{P}^{r+1}$, where $2\leq m\leq n-1$ and $r=m(n+1-m)+n-1$.
\end{prop}

\begin{rmk} \label{describ VMRT of symplectic Grassmannian}
Let $V\subseteq\mathbb{P}^{r}$ be as in Proposition \ref{classification of conic-connected manifolds}$(ii)$. By \cite[Lem. 3.6]{FH12}, $V\subseteq\mathbb{P}^{r}$ is projectively equivalent to the VMRT of the symplectic Grassmannian $\text{Gr}_{\omega}(n-m, 2n+1-m)$ at a general point. Then $\dim(\text{Sec}(V))=2n$ by \cite[Lem .4.19]{FH12}. Hence, either $\text{Sec}(V)=\mathbb{P}^{r}$ or $\dim(\text{Sec}(V))\leq r-2$. Furthermore, $\text{Sec}(V)=\mathbb{P}^{r}$ if and only if $m=n-2$, and in this case $V$ is a general hyperplane section of $\mathbb{P}^{1}\times\mathbb{P}^{n-1}\subseteq\mathbb{P}^{2n-1}$  by \cite[Lem. 3.7]{FH12}.
\end{rmk}

Note that QEL-manifolds with positive secant defects are conic-connected manifolds. As a direct consequence, we get the following

\begin{cor} \label{classification QEL (not prime Fano) when Sec(V)=P^r and delta>0}
Let $V\subseteq\mathbb{P}^{r}$ be a QEL-manifold with secant defect $\delta>0$. Assume that $\text{Sec}(V)=\mathbb{P}^{r}$. Then either $V\subseteq\mathbb{P}^{r}$ is a prime Fano manifold, or it is projectively equivalent to one of the following:

$(a)$ a smooth conic in $\mathbb{P}^{2}$;

$(b)$ the Segre embedding $\mathbb{P}^{1}\times\mathbb{P}^{n-1}\subseteq\mathbb{P}^{2n-1}$;

$(c)$ a nonsingular hyperplane section of $(b)$.
\end{cor}


\section{\normalsize Quadratic special birational transformations} \label{section: type (2, b)}

Now we fix some notations in this paper. Let $\phi: \mathbb{P}^{r}\dashrightarrow Z$ be a special birational transformation of type $(2, b)$. We always assume $b\geq 2$, and $Z\subsetneq\mathbb{P}^{N}$ is a nondegenerate prime Fano manifold. Let $X$ (resp. $Y$) be the base locus scheme of $\phi$ (resp. $\phi^{-1}$). Denote by $n$ (resp. $m$) the dimension of $X$ (resp. $Y$), $\delta$ the secant defect of $X$, and $c=N-r$. Let $\sigma: W\rightarrow\mathbb{P}^{r}$ be the blow-up of $\mathbb{P}^{r}$ along $X$. There is a natural morphism $\tau: W\rightarrow Z$ such that $\phi=\tau\circ\sigma^{-1}$. Denote by $E_{\mathbb{P}}=\sigma^{-1}(X)$, $E_{Z}=\tau^{-1}(Y)$ (scheme theoretically), $D_{\mathbb{P}}=\tau(E_{\mathbb{P}})$, and $H_{\mathbb{P}}\subseteq\mathbb{P}^{r}$ (resp. $H_{Z}\subseteq Z$) a hyperplane section.

\subsection{\normalsize Properties of $X$ and $Z^{(1)}$}

Now we recall some facts on $X$, most of which are from \cite[Prop. 1.3]{ES89} and \cite[Prop. 1, Prop. 2, Prop. 3, Rmk. 2]{AS13a}, except the linear normality of $X$ is $\mathbb{P}^{r}$ follows from \cite[Prop. 4.4(1)]{Sta12}.

\begin{prop} \label{X is QEL and Sec(X) is hypersurface}

$(i)$ $\text{Sec}(X)=\sigma(E_{Z})\subseteq\mathbb{P}^{r}$ is a hypersurface of degree $2b-1$, and $r=2n+2-\delta$;

$(ii)$ $X\subseteq\mathbb{P}^{r}$ is a nondegenerate linearly normal quadratic QEL-manifold;

$(iii)$ $\sigma^{*}H_{\mathbb{P}}=b\tau^{*}H_{Z}-E_{Z}$, and $\tau^{*}H_{Z}=2\sigma^{*}H_{\mathbb{P}}-E_{\mathbb{P}}$ in $\text{Pic}(W)$;

$(iv)$ $m=2n-2\delta$.
\end{prop}

\begin{prop} (\cite[Prop. 2.12]{AS13b}) \label{X linearly embedded into L^1(Z)}
Take a general point $z\in Z$. Let $p\in\phi^{-1}(z)$. Denote by $\pi_{p}: \mathbb{P}^{r}\dashrightarrow\mathbb{P}^{r-1}=\mathbb{P}((T_{z}Z)^{*})$ the linear projection from $p$. Then $\pi_{p}$ has definition on every point of $X$, $\pi_{p}(X)\subseteq Z^{(1)}$ and $\pi_{p}|_{X}: X\rightarrow\pi_{p}(X)$ is an isomorphism.
\end{prop}

\begin{prop} (\cite[Cor. 2.15]{AS13b}) \label{L_z smooth and (when b=2) covered by lines}

$(i)$ $Z^{(1)}\subseteq\mathbb{P}^{r-1}$ is smooth irreducible and non-degenerate;

$(ii)$ $\dim(Z^{(1)})=n+(b-1)(\delta+1)$;

$(iii)$ if $b=2$, then $Z^{(1)}\subseteq\mathbb{P}^{r-1}$ is covered by lines.
\end{prop}

A projective variety $V\subseteq\mathbb{P}^{r}$ is called a \textit{complete intersection of type $(d_{1}, \ldots, d_{c})$}, if it is the complete intersection of $c$ independent hypersurfaces of degrees $d_{1}, \ldots, d_{c}$, respectively.

\begin{prop} (\cite[Thm. 2.4(3)]{IR13}) \label{VMRT of complete intersection is also complete intersection}
Let $V\subseteq\mathbb{P}^{r}$ be a nondegenerate smooth complete intersection of type $(d_{1}, \ldots, d_{c})$. Assume that $V$ is covered by lines. Then $V^{(1)}\subseteq\mathbb{P}^{r-c-1}$ is a smooth complete intersection of type $(2, 3, \ldots, d_{1}, \ldots, 2, 3, \ldots, d_{c})$.
\end{prop}

\begin{prop} (\cite[Thm. 1.1]{Li14}) \label{my result: deg(var)|deg(subvar) and codim_spansub(sub)>=codim_spanvar(var)}
Let $V\subsetneq\mathbb{P}^{r}$ be an $n$-dimensional nondegenerate smooth projective variety and let $S\subseteq V$ be an $s$-dimensional closed subvariety. Assume that either

$(i)$ $s\geq \frac{r}{2}$; or

$(ii)$ $V$ is a complete intersection in $\mathbb{P}^{r}$, and $s>\frac{n}{2}$.

Then $\deg(V)$ divides $\deg(S)$.
\end{prop}

As an application of previous results, we have the following

\begin{prop} \label{deg(L_z) divides deg(X)}
Assume that either

$(i)$ $\delta\geq 1$; or

$(ii)$ $Z\subsetneq\mathbb{P}^{N}$ is a nondegenerate smooth complete intersection of type $(d_{1}, \ldots, d_{c})$ such that $(c, d_{1})\neq (1, 2)$.

Then $\deg(Z^{(1)})$ divides $\deg(X)$.
\end{prop}

\begin{proof}
By Proposition \ref{L_z smooth and (when b=2) covered by lines}$(i)$, $Z^{(1)}$ is an irreducible smooth closed subvariety of $\mathbb{P}^{r-1}$. By Proposition \ref{X linearly embedded into L^1(Z)}, there is a linear projection on $\mathbb{P}^{r}$ sending $X$ is isomorphically to a closed subvariety $X'$ of $Z^{(1)}$. In particular, $\dim(X')=\dim(X)$ and $\deg(X')=\deg(X)$.

$(i)$ By Proposition \ref{X is QEL and Sec(X) is hypersurface}$(i)$, $2\dim(X)=2n\geq r-1$. Thus, $\deg(Z^{(1)})$ divides $\deg(X')=\deg(X)$ by Proposition \ref{my result: deg(var)|deg(subvar) and codim_spansub(sub)>=codim_spanvar(var)}.

$(ii)$ Let $d=\sum\limits_{i=1}^{r}(d_{i}-1)$. By assumption, $d\geq 2$. By $(i)$, we can assume that $\delta=0$. By Proposition \ref{VMRT of complete intersection is also complete intersection}, $Z^{(1)}\subseteq\mathbb{P}^{r-1}$ is a smooth complete intersection of dimension $r-d-1\leq r-3$. Then $\dim(Z^{(1)})< 2\dim(X)$ by Proposition \ref{X is QEL and Sec(X) is hypersurface}$(i)$. Thus, $\deg(Z^{(1)})$ divides $\deg(X')=\deg(X)$ by Proposition \ref{my result: deg(var)|deg(subvar) and codim_spansub(sub)>=codim_spanvar(var)}.
\end{proof}

\subsection{\normalsize Properties of $Y$}

\begin{prop} \label{D is divisor degree 2b-1 and Y<=D<=Sec(Y)}
$D_{\mathbb{P}}\subseteq Z$ is a divisor, $\deg(D_{\mathbb{P}})=(2b-1)\deg(Z)$, and $Y\subseteq D_{\mathbb{P}}\subseteq\text{Sec}(Y)$.
\end{prop}

\begin{proof}
The discussion is an analogue of \cite[Prop. 2.3]{ES89}. Remark that $\tau|_{E_{\mathbb{P}}}: E_{\mathbb{P}}\rightarrow D_{\mathbb{P}}$ is birational, and $D_{\mathbb{P}}\subseteq Z$ is a divisor. By the projection formula and Proposition \ref{X is QEL and Sec(X) is hypersurface}$(iii)$,
$$\deg(D_{\mathbb{P}})=D_{\mathbb{P}}\cdot H_{Z}^{r-1}=E_{\mathbb{P}}\cdot(\tau^{*}H_{Z})^{r-1}=((2b-1)\tau^{*}H_{Z}-2E_{Z})\cdot(\tau^{*}H_{Z})^{r-1}=(2b-1)\deg(Z).$$

Take a general point $y\in Y$ and an irreducible curve $C\subseteq\tau^{-1}(y)$. Then
$$E_{\mathbb{P}}\cdot C=((2b-1)\tau^{*}H_{Z}-2E_{Z})\cdot C=-2E_{Z}\cdot C>0.$$
In particular, $E_{\mathbb{P}}\cap\tau^{-1}(y)\neq\emptyset$. Hence, $y\in\tau(E_{\mathbb{P}})$ and $Y\subseteq \tau(E_{\mathbb{P}})=D_{\mathbb{P}}$.

Let $z\in D_{\mathbb{P}}$ be a general point. Then $\tau^{-1}(z)=\{w\}\subseteq E_{\mathbb{P}}$. Moreover, $L=\sigma^{-1}\sigma(w)$ is a linear space contained in $E_{\mathbb{P}}$, and $\tau|_{L}: L\rightarrow\tau(L)$ is birational. Take a general line $l$ in $L$ passing through $w$. Then
\begin{eqnarray*}
&&E_{Z}\cdot l=((2b-1)\sigma^{*}H_{\mathbb{P}}-bE_{\mathbb{P}})\cdot l=b, \\
&&H_{Z}\cdot\tau(l)=\tau^{*}H_{Z}\cdot l=(2\sigma^{*}H_{\mathbb{P}}-E_{\mathbb{P}})\cdot l=1.
\end{eqnarray*}
Thus, $\tau(l)$ is a line contained in $D_{\mathbb{P}}$ passing through $z$, and it intersects $Y$ at $b$ points. By assumption, $b\geq 2$. Hence, $z\in\text{Sec}(Y)$ and $D_{\mathbb{P}}\subseteq\text{Sec}(Y)$.
\end{proof}

\begin{prop} \label{Y is nondegenerate}
$Y$ is nondegenerate in $\mathbb{P}^{N}$.
\end{prop}

\begin{proof}
By Proposition \ref{D is divisor degree 2b-1 and Y<=D<=Sec(Y)}, it suffices to show that $D_{\mathbb{P}}$ is nondegenerate in $\mathbb{P}^{N}$. We assume that there is a hyperplane $\mathbb{P}^{N-1}\subseteq\mathbb{P}^{N}$ containing $D_{\mathbb{P}}$. Then $\mathbb{P}^{N-1}\cap Z$ is of pure dimension $r-1$, and $\deg(D_{\mathbb{P}})\leq\deg(\mathbb{P}^{N-1}\cap Z)\leq\deg(Z)$. This contradicts Proposition \ref{D is divisor degree 2b-1 and Y<=D<=Sec(Y)}, since $b\geq 2$.
\end{proof}

Let $V\subseteq\mathbb{P}^{N}$ be a projective variety. Denote by $S_{V}$ the closure of the set of triples $(v_{1}, v_{2}, u)$ in $V\times V\times\mathbb{P}^{N}$ such that $v_{1}$ and $v_{2}$ are distinct points and $u\in\langle v_{1}, v_{2}\rangle$. Let $p_{i}$ be the $i$-th projection from $S_{V}$, and $p_{12}: S_{V}\rightarrow V\times V$ the projection to $V\times V$. For any point $u\in\text{Sec}(V)\backslash V$, denote by $\widetilde{\Sigma}_{u}=p_{1}(p_{3}^{-1}(u))$.
Remark that when $u\in\text{Sec}(V)\backslash V$ is a general point, we have $\widetilde{\Sigma}_{u}=\Sigma_{u}$.

\begin{lem} \label{C_z intersects Z in the union of Y and z}
Assume $b=2$. For any point (if exists) $z\in (\text{Sec}(Y)\cap Z)\backslash D_{\mathbb{P}}$, we have $C_{z}\cap Z\subseteq Y\cup\{z\}$. In particular, $\{z\}$ is the unique irreducible component of $C_{z}\cap Z$ containing $z$.
\end{lem}

\begin{proof}
By Proposition \ref{D is divisor degree 2b-1 and Y<=D<=Sec(Y)}, $Y\subseteq D_{\mathbb{P}}$ and $z\notin Y$. Hence, $\widetilde{\Sigma}_{z}$ is well-defined. Take an arbitrary point $y\in\widetilde{\Sigma}_{z}$. Then the line $l=\langle y, z\rangle$ is either a secant line or a tangent line of $Y$.

To get the conclusion, it suffices to show $l\cap Z\subseteq Y\cup\{z\}$. Assume that there is a point $z'\in (l\cap Z)\backslash (Y\cup\{z\})$. Then $\phi^{-1}(z)=\phi^{-1}(z')$, since $\phi^{-1}$ is defined by quadric hypersurfaces, and $l$ is either a secant line or a limit of secant lines of $Y$. This implies that $\phi^{-1}(z)$ lies in the base locus $X$ of $\phi$. Hence, $z\in\tau\sigma^{-1}(X)=D_{\mathbb{P}}$, which contradicts the choice of $z$.
\end{proof}

\begin{cor} \label{D is the intersection of Z and Sec(Y)}
Assume $b=2$. If $Z$ is a quadratic manifold, then $\text{Sec}(Y)\cap Z=D_{\mathbb{P}}$.
\end{cor}

\begin{proof}
Take any point $z\in (\text{Sec}(Y)\cap Z)\backslash Y$ and any point $y\in\widetilde{\Sigma}_{z}$. Then the line $l=\langle y, z\rangle$ is either a secant line or a tangent line of $Y$. Since $Z$ is a quadratic manifold, and $l$ intersects $Z$ at three or more points (counted with multiplicities), we get that $l\subseteq Z$. Since $\phi^{-1}$ is defined by quadric hypersurfaces, $\phi^{-1}(l\backslash (l\cap Y))$ is a point. Then $l\subseteq\tau\sigma^{-1}(X)=D_{\mathbb{P}}$. Thus, $\text{Sec}(Y)\cap Z \subseteq D_{\mathbb{P}}\cup Y$. By Proposition \ref{D is divisor degree 2b-1 and Y<=D<=Sec(Y)}, $Y\subseteq D_{\mathbb{P}}\subseteq\text{Sec}(Y)$. Hence, $\text{Sec}(Y)\cap Z=D_{\mathbb{P}}$.
\end{proof}

\begin{lem} \label{c>= n-2 delta}
Assume $b=2$. If $Y$ is smooth, then $c\geq n-2\delta$.
\end{lem}

\begin{proof}
By Proposition \ref{X is QEL and Sec(X) is hypersurface}$(i)(iv)$, $r=2n+2-\delta$ and $m=2n-2\delta$. Assume that $c<n-2\delta$. Then $m>\frac{2}{3}(N-2)$. Since $Y$ is smooth and nondegenerate, the secant variety $\text{Sec}(Y)=\mathbb{P}^{N}$ by Proposition \ref{sec(V) neq P^r then delta<=n/2}. In particular, $Z\subseteq\text{Sec}(Y)$. Take an arbitrary point $z\in Z\backslash D_{\mathbb{P}}$. By Lemma \ref{C_z intersects Z in the union of Y and z}, $\{z\}$ is the unique irreducible component of $C_{z}\cap Z$ containing $z$.
Let $M$ be a maximal dimensional irreducible component of $C_{z}$. Then
\begin{eqnarray*}
\dim(M)=\dim(C_{z})=\dim(\Sigma_{z})+1\geq\delta(Y)+1=2m+2-\dim(\text{Sec}(Y))=2n-3\delta-c.
\end{eqnarray*}
Since $C_{z}$ is a cone over the vertex $z$, we get $z\in M\cap Z$. Thus, $\{z\}$ is an irreducible component of $M\cap Z$. Hence,
\begin{eqnarray*}
\dim\{z\}\geq \dim(M)+\dim(Z)-\dim(\mathbb{P}^{N})\geq 2n-3\delta-2c>\delta\geq 0,
\end{eqnarray*}
which is a contradiction.
\end{proof}

\subsection{\normalsize Cases with large $\delta$}

Remark that $X$ is nondegenerate in $\mathbb{P}^{r}$ and $\text{Sec}(X)\neq\mathbb{P}^{r}$ by Proposition \ref{X is QEL and Sec(X) is hypersurface}$(i)$. Thus, $\delta\leq\frac{n}{2}$ by Proposition \ref{sec(V) neq P^r then delta<=n/2}. If $\delta=\frac{n}{2}$, then $X\subseteq\mathbb{P}^{r}$ is a Severi variety by definition. In this case, we know the following result due to Ein and Shepherd-Barron.

\begin{prop} (\cite[Thm. 2.6]{ES89}) \label{Z=P^r if and only if X is Severi}
Let $\Phi: \mathbb{P}^{r}\dashrightarrow\mathbb{P}^{r}$ be a special birational transformation of type $(a, b)$ with base locus scheme $V$.
Then $(a, b)=(2, 2)$ if and only if $V\subseteq\mathbb{P}^{r}$ is a Severi variety. Moreover, if $(a, b)=(2, 2)$ then $\Phi^{-1}$ is also a special birational transformation of type $(2, 2)$.
\end{prop}

\begin{prop} \label{delta<n/2}
We have $\delta<\frac{n}{2}$.
\end{prop}

\begin{proof}
By Proposition \ref{X is QEL and Sec(X) is hypersurface}$(i)(ii)$, $r=2n+2-\delta$, $X\subseteq\mathbb{P}^{r}$ is a nondegenerate quadratic QEL-manifold and $\text{Sec}(X)\neq\mathbb{P}^{r}$. Then $\delta\leq\frac{n}{2}$ by Proposition \ref{sec(V) neq P^r then delta<=n/2}. If $\delta=\frac{n}{2}$, then $X$ is a Severi variety by definition, and $Z=\mathbb{P}^{r}$ by Proposition \ref{Z=P^r if and only if X is Severi}. This contradicts our assumption $Z\neq\mathbb{P}^{r}$.
\end{proof}




\begin{prop} \label{classification of X when is not a prime Fano manifold}
Assume that $\delta>0$. Then either $X\subseteq\mathbb{P}^{r}$ is a prime Fano manifold with index $\frac{n+\delta}{2}$, or it is projectively equivalent to one of the following:

$(a)$ the second Veronese embedding $\upsilon_{2}(\mathbb{P}^{2})\subseteq\mathbb{P}^{5}$;

$(b)$ the Segre embedding $\mathbb{P}^{2}\times\mathbb{P}^{2}\subseteq\mathbb{P}^{8}$;

$(c)$ a nonsingular hyperplane of $(b)$;
\end{prop}

\begin{proof}
By Proposition \ref{X is QEL and Sec(X) is hypersurface}$(ii)$, $X$ is a nondegenerate linearly normal QEL-manifold. Thus, $X$ is conic-connected. By Proposition \ref{classification of conic-connected manifolds}, $X\subseteq\mathbb{P}^{r}$ is either a prime Fano manifold with index $i(X)=\frac{n+\delta}{2}$, or it is projectively equivalent to the cases $(i)-(iv)$ listed there. Now assume the latter case holds. By Proposition \ref{X is QEL and Sec(X) is hypersurface}$(i)$ $r=2n+2-\delta$. Hence, the only possible case in $(i)$ is $(a)$, the only possible case in $(iii)$ is $(b)$, and the only possible case in $(iv)$ is $(c)$. By Remark \ref{describ VMRT of symplectic Grassmannian}, there is no case in $(ii)$ satisfying $r=2n+2-\delta$.
\end{proof}


\begin{defi}
A QEL-manifold $V\subseteq\mathbb{P}^{r}$ is called a \textit{Special-QEL-manifold} (SQEL-manifold for short), if a general point $u\in\text{Sec}(V)\backslash V$ satisfies that for any point $u'\in C_{u}\backslash V$, we have $\Sigma_{u'}=\Sigma_{u}$.
\end{defi}

\begin{e.g.} \label{X is a SQEL-mfd e.g.}
By the proof of \cite[Prop. 2.3]{ES89} (see also \cite[Rmk. 2.4]{AS13b}), $X$ is a quadratic SQEL-manifold.
\end{e.g.}

The following Proposition will be proved in Section \ref{section: Proof of non-existence of certain 25-dim. quadratic QEL-manifolds}.

\begin{prop} \label{not exist 25-dim. quadric SQEL-mfd with delta=9}
There does not exist any nondegenerate $25$-dimensional quadratic SQEL-manifold $V\subseteq\mathbb{P}^{43}$ with secant defect $9$.
\end{prop}

Now we can rule out some cases with $\delta>\frac{n}{3}$.

\begin{prop} \label{case delta> n/3}
If $\delta>\frac{n}{3}$, then $(n, \delta)=(7, 3)$, $(10, 4)$, $(13, 5)$, $(14, 6)$, or $(15, 7)$.
\end{prop}

\begin{proof}
By Proposition \ref{delta<n/2}, $\frac{n}{3}<\delta<\frac{n}{2}$. Then by Proposition \ref{L^1(V) is QEL and Divisibility Theorem} $(ii)$, we get $\delta\leq 10$ and the following list for the possible values of the pair $(n, \delta)$: $(5, 2)$, $(7, 3)$, $(10, 4)$, $(13, 5)$, $(14, 6)$, $(15, 7)$, $(25, 9)$ and $(26, 10)$. One can also find this list in \cite[Prop. 3.6]{IR08}. Now it suffices to exclude the three cases not appearing in the conclusion.

Case 1. $(n, \delta, r)=(5, 2, 10)$: It is excluded by Proposition \ref{classification of X when is not a prime Fano manifold}, since $n+\delta$ is odd.

Case 2. $(n, \delta, r)=(25, 9, 43)$: By Example \ref{X is a SQEL-mfd e.g.}, $X$ is a quadratic SQEL-manifold. Then the existence of such $X$ contradicts Proposition \ref{not exist 25-dim. quadric SQEL-mfd with delta=9}.

Case 3. $(n, \delta, r)=(26, 10, 44)$. By Proposition \ref{L^1(V) is QEL and Divisibility Theorem}$(i)$, $X^{(1)}\subseteq\mathbb{P}^{25}$ is a QEL-manifold of dimension $16$ with secant defect $8$ and $\text{Sec}(X^{(1)})=\mathbb{P}^{25}$. However, such a QEL-manifold as $X^{(1)}$ does not exist by Proposition \ref{classification of QEL with delta=n/2}.
\end{proof}

\begin{rmk}
Let $\Phi: \mathbb{P}^{r+1}\dashrightarrow\mathbb{P}^{r+1}$ be a quadro-quadric special birational transformation. By Proposition \ref{Z=P^r if and only if X is Severi}, if $X\subseteq\mathbb{P}^{r}$ is the section of the base locus scheme of $\Phi$ by a general hyperplane $H\subseteq\mathbb{P}^{r+1}$, then $\phi=\Phi|_{H}: H\dashrightarrow \Phi(H)$ is a birational map defined by quadric hypersurfaces. Moreover, $\Phi(H)$ is a smooth quadric hypersurface (see for example \cite[Example 5.1]{Sta12}). This shows the possibility of $(n, \delta)=(7, 3)$ and $(15, 7)$ in Proposition \ref{case delta> n/3}. For the rest three cases there, we believe they are impossible. When $Z\subseteq\mathbb{P}^{r+1}$ is a nondegenerate smooth hypersurface, we can rule out them (see Proposition \ref{case d<=5 and case delta>n/3} in the following).
\end{rmk}

\section{\normalsize Quadro-quadric special birational transformations to smooth complete intersections}  \label{section: proof of main theorems}

\emph{Unless otherwise stated, we assume throughout this section that $b=2$, and $Z\subsetneq\mathbb{P}^{r+c}$ is a nondegenerate smooth complete intersection of type $(d_{1}, \ldots, d_{c})$, where $d_{1}\geq \cdots \geq d_{c}\geq 2$.}

\subsection{\normalsize Basic formulae}

\begin{prop} \label{delta>0 if Z complete intersection and n>1}
Assume that $n\geq 2$. Then $\delta>0$ and $\text{Pic}(Z^{(1)})=\mathbb{Z}(\mathcal{O}_{Z^{(1)}}(1))$.
\end{prop}

\begin{proof}
By Proposition \ref{L_z smooth and (when b=2) covered by lines}$(i)(ii)$, $Z^{(1)}$ is an irreducible smooth projective variety of dimension $n+\delta+1\geq 3$. By Proposition \ref{VMRT of complete intersection is also complete intersection}, $Z^{(1)}\subseteq\mathbb{P}^{r-1}$ is a complete intersection. Then $H^{0}(\mathbb{P}^{r-1}, \mathcal{O}_{\mathbb{P}^{r-1}}(k))\rightarrow H^{0}(Z^{(1)}, \mathcal{O}_{Z^{(1)}}(k))$ is surjective for any integer $k$. By the Lefchetz Theorem on complete intersections, $\text{Pic}(Z^{(1)})=\mathbb{Z}(\mathcal{O}_{Z^{(1)}}(1))$.

Now we assume that $\delta=0$. By Proposition \ref{X linearly embedded into L^1(Z)}, there is a linear projection $\pi:\mathbb{P}^{r}\dashrightarrow\mathbb{P}^{r-1}$ such that $X$ is isomorphically mapped to a closed subvariety $X'$ of $Z^{(1)}$. By Proposition \ref{X is QEL and Sec(X) is hypersurface}$(i)$, $X$ is nondegenerate in $\mathbb{P}^{r}$. Hence, $H^{0}(\mathbb{P}^{r-1}, \mathcal{O}_{\mathbb{P}^{r-1}}(1))\rightarrow H^{0}(X', \mathcal{O}_{X'}(1))$ is not surjective.

Since $\dim(Z^{(1)})=\dim(X)+\delta+1=\dim(X')+1$ and $\text{Pic}(Z^{(1)})=\mathbb{Z}(\mathcal{O}_{Z^{(1)}}(1))$, we know that $\mathcal{O}_{Z^{(1)}}(X')=\mathcal{O}_{Z^{(1)}}(k_{0})$ for some positive integer $k_{0}$. The surjection of the map $H^{0}(\mathbb{P}^{r-1}, \mathcal{O}_{\mathbb{P}^{r-1}}(k_{0}))\rightarrow H^{0}(Z^{(1)}, \mathcal{O}_{Z^{(1)}}(k_{0}))$ implies that $X'$ is the scheme-theoretic intersection of $Z^{(1)}$ and a hypersurface of degree $k_{0}$ in $\mathbb{P}^{r-1}$. Hence, $X'$ is a smooth complete intersection in $\mathbb{P}^{r-1}$. This implies that $H^{0}(\mathbb{P}^{r-1}, \mathcal{O}_{\mathbb{P}^{r-1}}(1))\rightarrow H^{0}(X', \mathcal{O}_{X'}(1))$ is surjective, which is a contradiction.
\end{proof}

\begin{prop} \label{delta=(n+c-sum d_i)/2 and sum (d_i^2)<=3n+c}
We have $\sum\limits_{i=1}^{c}d_{i}=n+c-2\delta$ and $\sum\limits_{i=1}^{c}d_{i}^{2}\leq 3n+c$.
In particular, $\delta\geq\frac{n+c-\sqrt{c(3n+c)}}{2}$.
\end{prop}

\begin{proof}
By Proposition \ref{L_z smooth and (when b=2) covered by lines}$(ii)$, $\dim(Z^{(1)})=n+\delta+1$. On the other hand, by Proposition \ref{VMRT of complete intersection is also complete intersection}, $\dim(Z^{(1)})=r-1+c-\sum\limits_{i=1}^{c}d_{i}$. Thus, $r=\sum\limits_{i=1}^{c}d_{i}+n+\delta+2-c$. By Proposition \ref{X is QEL and Sec(X) is hypersurface}$(i)$, $r=2n+2-\delta$. Hence, $\sum\limits_{i=1}^{c}d_{i}=n+c-2\delta$.

If $n=1$, then by the formula $\sum\limits_{i=1}^{c}d_{i}=n+c-2\delta$, we get that $\delta=0$, $c=1$ and $d_{1}=2$. In particular, $\sum\limits_{i=1}^{c}d_{i}^{2}\leq 3n+c$ holds when $n=1$. Now we assume that $n\geq 2$. Then by Proposition \ref{delta>0 if Z complete intersection and n>1}, $\delta>0$ and $\text{Pic}(Z^{(1)})=\mathbb{Z}(\mathcal{O}_{Z^{(1)}}(1))$. Then $Z^{(1)}$ is a prime Fano manifold with index $i(Z^{(1)})\geq 2$ by Proposition \ref{L_z smooth and (when b=2) covered by lines}$(iii)$. By Proposition \ref{VMRT of complete intersection is also complete intersection}, the index $i(Z^{(1)})=r+c-\sum\limits_{i=1}^{c}\frac{d_{i}(d_{i}+1)}{2}$. Combining with $i(Z^{(1)})\geq 2$ and $\sum\limits_{i=1}^{c}d_{i}=n+c-2\delta$, we get $\sum\limits_{i=1}^{c}d_{i}^{2}\leq 3n+c$. Remark that $(\sum\limits_{i=1}^{c}d_{i})^{2}\leq c\sum\limits_{i=1}^{c}d_{i}^{2}$. Then we have $\delta=\frac{n+c-\sum\limits_{i=1}^{c}d_{i}}{2}\geq\frac{n+c-\sqrt{c(3n+c)}}{2}$.
\end{proof}

\begin{cor} \label{e_1=n-2delta-c, e_2<=... etc.}
Let $e_{1}=\sum\limits_{i=1}^{c}(d_{i}-2)$ and $e_{2}=\sum\limits_{i=1}^{c}(d_{i}-2)^{2}$. Then $e_{1}=n-2\delta-c$ and $e_{2}\leq 8\delta+c-n$. In particular, $c\leq n-2\delta$ and $\delta\geq\frac{n-c}{8}$.
\end{cor}

\begin{proof}
Remark that all $d_{i}\geq 2$. Then $e_{1}\geq 0$ and $e_{2}\geq 0$. So the conclusion is a direct consequence of Proposition \ref{delta=(n+c-sum d_i)/2 and sum (d_i^2)<=3n+c}.
\end{proof}

There is a classification of the cases where $c=1$ and $d_{1}=2$ due to Staglian\`{o} as follows:

\begin{prop} (\cite[Thm. 6.1]{Sta12}) \label{Z hyperquadric then Y Severi and X hyperplane section of Severi}
Assume that $c=1$ and $d_{1}=2$. Then $Y$ is a Severi variety and $X$ is a nonsingular hyperplane section of a Severi variety.
\end{prop}

As a consequence, we get the following

\begin{cor}
Assume $\delta=0$. Then $Z\subseteq\mathbb{P}^{5}$ is a $4$-dimensional nonsingular quadric hypersurface, $Y\subseteq\mathbb{P}^{5}$ is projectively equivalent to $\nu_{2}(\mathbb{P}^{2})\subseteq\mathbb{P}^{5}$ and $X\subseteq\mathbb{P}^{4}$ is projectively equivalent to a nonsingular hyperplane section of $\nu_{2}(\mathbb{P}^{2})\subseteq\mathbb{P}^{5}$.
\end{cor}

\begin{proof}
By Proposition \ref{delta>0 if Z complete intersection and n>1}, $n=1$. Then by Proposition \ref{delta=(n+c-sum d_i)/2 and sum (d_i^2)<=3n+c}, $c=1$ and $d_{1}=2$. So the conclusion follows from Proposition \ref{Z hyperquadric then Y Severi and X hyperplane section of Severi} and the classification of Severi varieties (see Proposition \ref{classification of Severi varieties}).
\end{proof}

\subsection{\normalsize Hypersurfaces cases}

In this subsection, we study the case where $Z\subseteq\mathbb{P}^{N}$ is a hypersurface of degree $d$, namely $c=1$ and $d_{1}=d$. This has been studied by Staglian\`{o} in \cite{Sta12}. We need the following result from him.

\begin{prop} (\cite[Prop. 4.4]{Sta12}) \label{Hilbert polynomial}
Assume $\delta>0$, and $X$ is not a nonsingular hyperplane section of $\mathbb{P}^{2}\times\mathbb{P}^{2}\subseteq\mathbb{P}^{8}$. Then the following hold.

$(i)$ $X$ is a prime Fano manifold with index $i(X)=\frac{n+\delta}{2}$.

$(ii)$ The Hilbert polynomial $P_{X}(t)$ of $X$ satisfies that $P_{X}(0)=1$, $P_{X}(1)=r+1$, $P_{X}(2)=\frac{r(r+1)}{2}-1$, $P_{X}(t)=0$ for $-i(X)+1\leq t\leq -1$, and $P_{X}(t)=(-1)^{n}P_{X}(-t-i(X))$ for all $t\in\mathbb{Z}$. In particular, when the coindex $c(X):=n+1-i(X)\leq 5$, $P_{X}(t)$ is uniquely determined.
\end{prop}

\begin{e.g.} \label{example Hilbert poly.}
Now we compute the Hilbert polynomial $P_{X}(t)$ for some possible values of $(n, \delta)$ with $\delta>0$.

If $(n, \delta)=(8, 2)$. Then $r=16$ by Proposition \ref{X is QEL and Sec(X) is hypersurface}$(i)$.
By Proposition \ref{Hilbert polynomial}, $X$ is prime Fano of index $5$. Let $Q(t)=P_{X}(\frac{t-5}{2})$. Then $Q(t)$ is an even polynomial function of degree $8$ with four roots $\pm 1$ and $\pm 3$. Thus, we can denote by $Q(t)=(t^{2}-1)(t^{2}-9)(a_{4}t^{4}+a_{2}t^{2}+a_{0})$. Since $Q(5)=1$, $Q(7)=17$ and $Q(9)=135$, we get $(a_{4}, a_{2}, a_{0})=(\frac{36}{2^{8}\cdot 8!}, \frac{24}{2^{8}\cdot 8!}, \frac{3780}{2^{8}\cdot 8!})$. Thus, $\deg(X)=36$ and $$P_{X}(t)=Q(2t+5)=\frac{1}{8!}(\prod\limits_{i=1}^{4}(t+i))(36t^{4}+360t^{3}+1374t^{2}+1245t+1680).$$
Similarly, if $(n, \delta)=(12, 4)$, then $\deg(X)=84$ and
$$P_{X}(t)=\frac{1}{12!}(\prod\limits_{i=1}^{7}(t+i))(t+4)(84t^{4}+1344t^{3}+8052t^{2}+21408t+23760).$$

\end{e.g.}


\begin{prop} \label{case d<=5 and case delta>n/3}
If either $d\leq 5$ or $\delta>\frac{n}{3}$, then $d=2$.
\end{prop}

\begin{proof}
If $\delta>\frac{n}{3}$, then $(n, \delta)=(7, 3)$, $(10, 4)$, $(13, 5)$, $(14, 6)$, or $(15, 7)$ by Proposition \ref{case delta> n/3}. By Proposition \ref{delta=(n+c-sum d_i)/2 and sum (d_i^2)<=3n+c}, the corresponding $d$ is $2, 3, 4, 3$, or $2$ respectively. Thus, we only need to exclude the cases with $3\leq d\leq 5$. Remark that $\deg(Z^{(1)})=d!$ by Proposition \ref{VMRT of complete intersection is also complete intersection}. If $d\geq 3$, then $d!$ divides $\deg(X)$ by Proposition \ref{deg(L_z) divides deg(X)}.

Case 1. Assume $d=5$. By the discussion above, $\delta\leq\frac{n}{3}$. By Proposition \ref{delta=(n+c-sum d_i)/2 and sum (d_i^2)<=3n+c}, $8\leq n\leq 12$, and $\delta=\frac{n}{2}-2$. By Proposition \ref{Hilbert polynomial}$(i)$, $n\neq 9, 10$, or $11$. Then  $(n, \delta)=(8, 2)$ or $(12, 4)$. By Example \ref{example Hilbert poly.}, $\deg(X)=36$ or $84$ respectively. However, $\deg(Z^{(1)})=120$, and it divides neither $36$ nor $84$. We get a contradiction.

Case 2. Assume $d=4$. By \cite[Prop. 8.3]{Sta12}, the only possible values of $(n, \delta, r, \deg(X))$ are $(9, 3, 17, 35)$ and $(13, 5, 23, 82)$. On the other hand, $\deg(Z^{(1)})=24$, and it divides neither $35$ nor $82$. It is a contradiction.

Case 3. Assume $d=3$. The only possible values of $(n, \delta, r, \deg(X))$ are $(10, 4, 18, 34)$ and $(14, 6, 24, 80)$ by \cite[Prop. 8.2]{Sta12}.  However, $\deg(Z^{(1)})=6$, and it divides neither $34$ nor $80$. It is a contradiction.
\end{proof}

\begin{cor} \label{case delta<=1}
If $\delta\leq 2$, then $d=2$.
\end{cor}

\begin{proof}
By Proposition \ref{delta=(n+c-sum d_i)/2 and sum (d_i^2)<=3n+c}, $n\leq 8$ and $d\leq 5$. Hence, $d=2$ by Proposition \ref{case d<=5 and case delta>n/3}.
\end{proof}

Now we are ready to prove Theorem \ref{introduction: thm. Z hypersurface cases}.

\begin{proof}[Proof of Theorem \ref{introduction: thm. Z hypersurface cases}]
By Proposition \ref{Z hyperquadric then Y Severi and X hyperplane section of Severi}, it suffices to show $d=2$. Now assume that $d\geq 3$. By Proposition \ref{delta=(n+c-sum d_i)/2 and sum (d_i^2)<=3n+c}, $d=n+1-2\delta$. Since $d\geq 3$, $n\neq\delta$. By Corollary \ref{case delta<=1}, $\delta\geq 3$. Then by Proposition \ref{L^1(V) is QEL and Divisibility Theorem} $(iii)$, $\delta\leq 2[log_{2}n]+2$. On the other hand, $\delta\geq\frac{n+1-\sqrt{3n+1}}{2}$ by Proposition \ref{delta=(n+c-sum d_i)/2 and sum (d_i^2)<=3n+c}. Hence, $n\leq 33$ and $\delta\leq 12$.

By Proposition \ref{L^1(V) is QEL and Divisibility Theorem} $(ii)$ and the formulae $d=n+1-2\delta$ and $d^{2}\leq 3n+1$ in Proposition \ref{delta=(n+c-sum d_i)/2 and sum (d_i^2)<=3n+c}, the possible values of $(n, \delta, d)$ are $(26, 10, 7)$, $(25, 9, 8)$, $(15, 7, 2)$, $(18, 6, 7)$, $(14, 6, 3)$, $(13, 5, 4)$, $(12, 4, 5)$, $(10, 4, 3)$, $(9, 3, 4)$ and $(7, 3, 2)$. By Proposition \ref{case d<=5 and case delta>n/3}, only the case $(n, \delta, d)=(18, 6, 7)$ is possible.

By Proposition \ref{L^1(V) is QEL and Divisibility Theorem} $(i)$, $X^{(1)}\subseteq\mathbb{P}^{17}$ is a QEL-manifold of dimension $10$ such that $\delta(X^{(1)})=4$ and $\text{Sec}(X^{(1)})=\mathbb{P}^{17}$, and $X^{(2)}\subseteq\mathbb{P}^{9}$ is a QEL-manifold of dimension $5$ such that $\delta(X^{(2)})=2$ and $\text{Sec}(X^{(2)})=\mathbb{P}^{9}$. By Proposition \ref{L^1(V) is QEL and Divisibility Theorem}$(iv)$, $X^{(2)}$ is not a prime Fano manifold. Then $X^{(2)}\subseteq\mathbb{P}^{9}$ is projectively equivalent to $\mathbb{P}^{1}\times\mathbb{P}^{4}\subseteq\mathbb{P}^{9}$ by Corollary \ref{classification QEL (not prime Fano) when Sec(V)=P^r and delta>0}.
Since the VMRT of $X^{(1)}$ at a general point is projectively equivalent to the VMRT of $\mathbb{G}(1, 6)$ at a general point, we know that $X^{(1)}$ is isomorphic to $\mathbb{G}(1, 6)$ (see for example \cite[Main Thm.]{Mok08}). Since both $X^{(1)}\subseteq\mathbb{P}^{17}$ and $\mathbb{G}(1, 6)\subseteq\mathbb{P}^{17}$ are covered by lines, the isomorphism between them are induced by a linear subsystem of $|\mathcal{O}_{X^{(1)}}(1)|$. Thus, $\dim(H^{0}(X^{(1)}, \mathcal{O}_{X^{(1)}}(1)))=\dim(H^{0}(\mathbb{G}(1, 6), \mathcal{O}_{\mathbb{G}(1, 6)}(1)))\geq 21$. On the other hand, the fact $\text{Sec}(X^{(1)})=\mathbb{P}^{17}$ implies that the QEL-manifold $X^{(1)}\subseteq\mathbb{P}^{17}$ is linearly normal (see Subsection \ref{subsection: QEL-manifolds}), which is a contradiction.
\end{proof}

\subsection{\normalsize When $Y$ is smooth}

Now we return to the case where $Z\subseteq\mathbb{P}^{N}$ is a complete intersection of type $(d_{1}, \ldots, d_{c})$. Firstly, we need a technique result on Severi varieties.

\begin{lem} \label{c. i. of type (2,...,2) containging Severi var. must be hypersurface}
Let $V\subseteq\mathbb{P}^{N}$ be a non-degenerate Severi variety, and $M\subseteq\mathbb{P}^{N}$ be a $c$-codimensional complete intersection of type $(2, \ldots, 2)$ containing $V$. If $M$ is smooth and connected, then $c=1$, i.e. $M$ is a quadric hypersurface containing $V$.
\end{lem}

In fact, the statement of this Lemma appears in \cite[Example 3.24$(i)$]{AS13b} without proof.

\begin{proof}
Assume that $M$ is smooth and connected, and $c\geq 2$. By Proposition \ref{classification of Severi varieties}, $V\subseteq\mathbb{P}^{N}$ is projectively equivalent to $\nu_{2}(\mathbb{P}^{2})\subseteq\mathbb{P}^{5}$, $\mathbb{P}^{2}\times\mathbb{P}^{2}\subseteq\mathbb{P}^{8}$, $\mathbb{G}(1, 5)\subseteq\mathbb{P}^{14}$, or $\mathbb{OP}^{2}\subseteq\mathbb{P}^{26}$. In all cases, $\dim(V)>\frac{\dim(M)}{2}$. By Proposition \ref{my result: deg(var)|deg(subvar) and codim_spansub(sub)>=codim_spanvar(var)}, $\deg(M)=2^{c}$ divides $\deg(V)$. Since $\deg(V)=4$, $6$, $14$, or $78$ in the corresponding four cases (see for example \cite[page 15-16]{Sta13}), we get that $c=2$, $V\subseteq\mathbb{P}^{N}$ is projectively equivalent to $\nu_{2}(\mathbb{P}^{2})\subseteq\mathbb{P}^{5}$ and $\deg(M)=\deg(V)$. By the isomorphism $\text{Pic}(\mathbb{P}^{N})\cong\text{Pic}(M)$ and the fact $V$ is a divisor on $M$ with the same degree as projective varieties, we know that $\mathcal{O}_{M}(V)=\mathcal{O}_{M}(1)$. Since the natural map $H^{0}(\mathbb{P}^{N}, \mathcal{O}_{\mathbb{P}^{N}}(1))\rightarrow H^{0}(M, \mathcal{O}_{M}(1))$ is surjective, there is a hyperplane $H$ in $\mathbb{P}^{N}$ such that $V=H\cap V$, which contradicts the fact that $V=\nu_{2}(\mathbb{P}^{2})\subseteq\mathbb{P}^{5}$ is non-degenerate. This finishes the proof.
\end{proof}

Now we are ready to prove Theorem \ref{introduction: thm. c.i. and Y smooth implies c=1 and deg(Z)=2}.

\begin{proof}[Proof of Theorem \ref{introduction: thm. c.i. and Y smooth implies c=1 and deg(Z)=2}]
By Lemma \ref{c>= n-2 delta}, $c\geq n-2\delta$. Then by Corollary \ref{e_1=n-2delta-c, e_2<=... etc.}, $c=n-2\delta$ and $d_{1}=\cdots=d_{c}=2$. In particular, $Z$ is a quadratic manifold. By Corollary \ref{D is the intersection of Z and Sec(Y)}, $\text{Sec}(Y)\cap Z=D_{\mathbb{P}}$. Hence, $\text{Sec}(Y)\neq\mathbb{P}^{N}$. Note that $m=\frac{2}{3}(N-2)$ by Proposition \ref{X is QEL and Sec(X) is hypersurface}$(i)(iv)$, and $Y$ is nondegenerate in $\mathbb{P}^{N}$ by Proposition \ref{Y is nondegenerate}. Hence, $Y\subseteq\mathbb{P}^{N}$ is a Severi variety by definition. By Lemma \ref{c. i. of type (2,...,2) containging Severi var. must be hypersurface}, $c=1$. The rest follows from Proposition \ref{Z hyperquadric then Y Severi and X hyperplane section of Severi}.
\end{proof}



Now we want to apply Theorem \ref{introduction: thm. c.i. and Y smooth implies c=1 and deg(Z)=2} to study the cases where $Z$ is a quadratic projective manifold. To do this, we need some properties on complete intersections.

\begin{prop} (\cite[Thm. 3.8(4)]{IR13}) \label{Hartshorne Conj. on quadratic mfds}
Let $V\subseteq\mathbb{P}^{r}$ be a quadratic manifold of dimension $n>\frac{2r}{3}$. Then $V$ is a complete intersection in $\mathbb{P}^{r}$.
\end{prop}

\begin{prop} (\cite[Thm. 2.4(4)]{IR13}) \label{V^1 c. i. then codim V <= codim V^1}
Let $V\subseteq\mathbb{P}^{r}$ be a nondegenerate  prime Fano manifold. Suppose $V^{(1)}\subseteq\mathbb{P}^{n-1}$ is a complete intersection such that $\dim(V^{(1)})\geq\frac{n-1}{2}$. Then $\text{codim}_{\mathbb{P}^{r}}(V)\leq\text{codim}_{\mathbb{P}^{n-1}}(V^{(1)})$.
\end{prop}

As a application of these results, we get the following

\begin{cor} \label{Z quadratic var. implies it quadric hypersurface}
Assume that $r\leq\frac{9}{5}n+2$, $Y$ is smooth, and $Z\subsetneq\mathbb{P}^{N}$ is a nondegenerate prime Fano quadratic projective manifold (we do not assume $Z$ to be a complete intersection). Then $Z$ is a quadric hypersurface, $Y$ is a Severi variety and $X$ is a hyperplane section of a Severi variety.
\end{cor}

\begin{proof}
Remark that $\dim(Z^{(1)})=n+\delta+1$ by Proposition \ref{L_z smooth and (when b=2) covered by lines}$(ii)$, and $r=2n+2-\delta$ by Proposition \ref{X is QEL and Sec(X) is hypersurface}$(i)$. Since $r\leq\frac{9}{5}n+2$, we have $\delta\geq\frac{n}{5}$. Hence, $\dim(Z^{(1)})>\frac{2}{3}(r-1)$. Moreover, $Z^{(1)}$ is a (possibly reducible) quadratic manifold (see Subsection \ref{subsection: quadratic manifolds}). By Proposition \ref{L_z smooth and (when b=2) covered by lines}$(i)$, $Z^{(1)}$ is irreducible. Then $Z^{(1)}$ is a complete intersection in $\mathbb{P}^{r-1}$ by Proposition \ref{Hartshorne Conj. on quadratic mfds}. By Proposition \ref{V^1 c. i. then codim V <= codim V^1}, $c\leq r-1-\dim(Z^{(1)})=n-2\delta$. Hence, $r>\frac{2N}{3}$. By Proposition \ref{Hartshorne Conj. on quadratic mfds}, $Z$ is a complete intersection. Then the conclusion follows from Theorem \ref{introduction: thm. c.i. and Y smooth implies c=1 and deg(Z)=2}.
\end{proof}





\section{\normalsize Proof of Proposition \ref{not exist 25-dim. quadric SQEL-mfd with delta=9}} \label{section: Proof of non-existence of certain 25-dim. quadratic QEL-manifolds}

Recall that a QEL-manifold $V\subseteq\mathbb{P}^{r}$ is called a SQEL-manifold, if a general point $u\in\text{Sec}(V)\backslash V$ satisfies that for any point $u'\in C_{u}\backslash V$, we have $\Sigma_{u'}=\Sigma_{u}$. To complete our proof of Theorem \ref{introduction: thm. Z hypersurface cases}, we need to prove Proposition \ref{not exist 25-dim. quadric SQEL-mfd with delta=9}. Our aim in this section is to prove Proposition \ref{not exist 25-dim. quadric SQEL-mfd with delta=9}, which claims the non-existence of nondegenerate $25$-dimensional quadratic SQEL-manifolds in $\mathbb{P}^{43}$ with secant defect $9$.

In Subsection \ref{subsection: general entry loci}, we study the properties of general entry loci on a SQEL-manifold and prove Proposition \ref{not exist 25-dim. quadric SQEL-mfd with delta=9} assuming the following Proposition \ref{not exist 15-dim. quadratic SQEL-mfd s.t. Sec=P^24}. Then we prove Proposition \ref{not exist 15-dim. quadratic SQEL-mfd s.t. Sec=P^24} in Subsection \ref{subsection: Proof of non-existence of certain 15-dim. quadratic QEL-manifolds}, which also requires a detailed study of entry loci on SQEL-manifold.

\begin{prop} \label{not exist 15-dim. quadratic SQEL-mfd s.t. Sec=P^24}
There does not exist any $15$-dimensional quadratic SQEL-manifold $V\subseteq\mathbb{P}^{24}$ such that $\text{Sec}(V)=\mathbb{P}^{24}$.
\end{prop}

\subsection{\normalsize General entry loci} \label{subsection: general entry loci}

\emph{Throughout this subsection, we assume that $V\subseteq\mathbb{P}^{r}$ is a nondegenerate SQEL-manifold of dimension $n$ such that the secant defect $1\leq\delta(V)< n$.}

Denote by $U_{g}(V)$ the set of points $u\in\text{Sec}(V)\backslash V$ such that the entry locus $\Sigma_{u}$ is an irreducible and smooth quadric hypersurface of dimension $\delta(V)$, and $\Sigma_{u'}=\Sigma_{u}$ for any point $u'\in C_{u}\backslash V$. Let $\mathcal{Q}_{g}(V)=\{\Sigma_{u}\mid u\in U_{g}(V)\}$. Since $V$ is a SQEL-manifold, we know that $U_{g}(V)$ contains a Zariski open dense subset $U^{o}(V)$ of $\text{Sec}(V)\backslash V$ and a general entry locus of $V$ belongs to $\mathcal{Q}_{g}(V)$.

\begin{lem} \label{C_u cap V=Sigma_u, C_(u_1) cap C_(u_2)=Sigma_(u_1) cap Sigma_(u_2) lem.}
Keeping notation as above. Take $u\in U_{g}(V)$. Then $C_{u}\cap V=\Sigma_{u}$. If $u'\in U_{g}(V)$ satisfies $\Sigma_{u'}\neq\Sigma_{u}$, then $\Sigma_{u}\cap\Sigma_{u'}=C_{u}\cap C_{u'}$.
\end{lem}

\begin{proof}
Remark that $\Sigma_{u}$ is a quadric hypersurface contained in $C_{u}=\mathbb{P}^{\delta(V)+1}$. Assume $C_{u}\cap V\neq\Sigma_{u}$. Then take $v\in (C_{u}\cap V)\backslash\Sigma_{u}$. There is a secant line $l$ of $\Sigma_{u}$ passing through $v$ not contained in $V$, since $C_{u}=\bigcup\limits_{v'\in\Sigma_{u}}\langle v, v'\rangle\nsubseteq V$. Take $u_{1}\in l\backslash V$. Then $v\in\Sigma_{u_{1}}=\Sigma_{u}$, where the equality follows from the fact $u\in U_{g}(V)$. This contradicts the choice of $v$. Hence, $C_{u}\cap V=\Sigma_{u}$.

Now assume $u'\in U_{g}(V)$ and $\Sigma_{u'}\neq\Sigma_{u}$. Then $C_{u}\cap C_{u'}\subseteq V$, since otherwise the existence of $u_{2}\in (C_{u}\cap C_{u'})\backslash V$ implies $\Sigma_{u}=\Sigma_{u_{2}}=\Sigma_{u'}$, which contradicts the choice of $\Sigma_{u}$ and $\Sigma_{u'}$. Hence, $C_{u}\cap C_{u'}=(C_{u}\cap V)\cap (C_{u'}\cap V)=\Sigma_{u}\cap\Sigma_{u'}$.
\end{proof}

\begin{e.g.} \label{entry loci on S^10 e.g.}
We consider the 10-dimensional Spinor variety $S^{10}\subseteq\mathbb{P}^{15}$. Note that the secant variety $\text{Sec}(S^{10})=\mathbb{P}^{15}$. It is known that $(a)$ each entry locus of $S^{10}\subseteq\mathbb{P}^{15}$ is a smooth connected quadric hypersurface of dimension $6$, and $(b)$ the intersection of any two different entry loci of $S^{10}\subseteq\mathbb{P}^{15}$ is either empty or a linear subspace of dimension $3$. By $(a)$, $U_{g}(S^{10})=\mathbb{P}^{15}\backslash S^{10}$.  Let $M$ be the section of $S^{10}$ by an arbitrary hyperplane $H$ in $\mathbb{P}^{15}$. Then by $(a)(b)$, any entry locus of $M$ is the intersection of $H$ and an entry locus of $S^{10}$, and the intersection of two different entry loci of $M$ is either empty or a linear subspace of dimension $2$ or $3$. For more details of this example, one can see \cite[Lem. 5.11, Prop. 5.12, Cor. 5.13]{FH14}.
\end{e.g.}

\begin{lem} \label{V^1 is a SQEL-mfd lem.}
Assume $\delta(V)\geq 3$. Take a general point $v\in V$. Then $\mathcal{L}_{v}(V)$ is a SQEL-manifold.
\end{lem}

\begin{proof}
Since $\delta(V)>0$, $\text{Sec}(V)$ is the closure of the union of embedded tangent spaces (see for example \cite[Thm. 1.4]{Zak93}). Hence, for the general point $v\in V$, $U^{o}(V)\cap\mathbb{T}_{v}V$ is an open dense subset in $\mathbb{T}_{v}V$. Take a general hyperplane $H$ of $\mathbb{T}_{v}V$. Then $U_{g}(V)\cap H\supseteq U^{o}(V)\cap H\neq\emptyset$ and $v\notin H$. We can identify $H$ with $\mathbb{P}((T_{v}V)^{*})$. Set $V^{(1)}=\mathcal{L}_{v}(V)$. By Proposition \ref{L^1(V) is QEL and Divisibility Theorem} $(i)$, $V^{(1)}\subseteq H$ is a QEL-manifold with secant defect $\delta(V)-2>0$ and $\text{Sec}(V^{(1)})=H$.

We claim that $U_{g}(V)\cap H\subseteq U_{g}(V^{(1)})$. If the claim holds, then $U_{g}(V^{(1)})$ contains an open dense subset $U^{o}(V)\cap H$ of $H$.  As a consequence $V^{(1)}$ is then a SQEL-manifold.

Now we turn to the proof the claim. Recall that $S_{V}$ is defined to be the closure of the set of triples $(v_{1}, v_{2}, u)$ in $V\times V\times\mathbb{P}^{r}$ such that $v_{1}\neq v_{2}$ and $u\in\langle v_{1}, v_{2}\rangle$. Let $p_{i}$ be the restriction to $S_{V}$ of the $i$-th projection from $V\times V\times\mathbb{P}^{r}$. Then $p_{3}(p_{1}^{-1}(v))$ is the joint variety of $v$ and $V$, i.e. it is the closure of the union of lines $\langle v, v_{1}\rangle$ for $v_{1}\in V\backslash\{v\}$. In particular, $\mathbb{T}_{v}V\subseteq p_{3}(p_{1}^{-1}(v))$. So for any $u_{0}\in\mathbb{T}_{v}V$, we have $v\in p_{1}(p_{3}^{-1}(u_{0}))$. Take $u\in U_{g}(V)\cap H\subseteq\mathbb{T}_{v}V$. Then $\Sigma_{u}(V)=p_{1}(p_{3}^{-1}(u))$ is an irreducible smooth $\delta(V)$-dimensional quadric hypersurface passing through $v$. Thus, $\Sigma_{u}(V)\cap H=\mathcal{L}_{v}(\Sigma_{u}(V))$ is an irreducible smooth $(\delta(V)-2)$-dimensional quadric hypersurface contained in $V^{(1)}$, which implies that $\Sigma_{u}(V)\cap H=\Sigma_{u}(V)\cap V^{(1)}$. On the other hand, $u\notin V^{(1)}$, since otherwise $u\in U_{g}(V)\cap V^{(1)}\subseteq U_{g}(V)\cap V=\emptyset$. By the definition of entry loci,
$$\Sigma_{u}(V^{(1)})\subseteq \Sigma_{u}(V)\cap V^{(1)}=\Sigma_{u}(V)\cap H.$$
Thus, $u\in C_{u}(V^{(1)})\subseteq C_{u}(V)\cap H$. Since $\Sigma_{u}(V)\cap V^{(1)}=\Sigma_{u}(V)\cap H$ is a quadric hypersurface in $C_{u}(V)\cap H$, we get that $\Sigma_{u}(V)\cap H\subseteq\Sigma_{u}(V^{(1)})$. Thus, $\Sigma_{u}(V^{(1)})=\Sigma_{u}(V)\cap H$ is an irreducible smooth $(\delta(V)-2)$-dimensional quadric hypersurface.

Now take any point $u'\in C_{u}(V^{(1)})\backslash V^{(1)}$, then $u'\in C_{u}(V)\cap H$. Moreover, $u'\notin V$, since otherwise
$$u'\in V\cap C_{u}(V^{(1)})\subseteq V\cap C_{u}(V)\cap H=\Sigma_{u}(V)\cap H=\Sigma_{u}(V^{(1)})\subseteq V^{(1)},$$
where the first equality follows from Lemma \ref{C_u cap V=Sigma_u, C_(u_1) cap C_(u_2)=Sigma_(u_1) cap Sigma_(u_2) lem.}. Since $u'\in C_{u}(V)\backslash V$ and $u\in U_{g}(V)$, we get that $\Sigma_{u'}(V)=\Sigma_{u}(V)$ and $C_{u'}(V)=C_{u}(V)$. This implies that $u'\in U_{g}(V)$. Hence, $u'\in U_{g}(V)\cap H$ and $$\Sigma_{u}(V^{(1)})=\Sigma_{u}(V)\cap H=\Sigma_{u'}(V)\cap H=\Sigma_{u'}(V^{(1)}),$$
where the third equality follows from the same argument as the first one (see the discussion in the last paragraph). Thus, $u, u'\in U_{g}(V^{(1)})$ and $U_{g}(V)\cap H \subseteq U_{g}(V^{(1)})$. So the claim holds.
\end{proof}

Now we can prove Proposition \ref{not exist 25-dim. quadric SQEL-mfd with delta=9}, assuming that Proposition \ref{not exist 15-dim. quadratic SQEL-mfd s.t. Sec=P^24} holds.

\begin{proof}[Proof of Proposition \ref{not exist 25-dim. quadric SQEL-mfd with delta=9}]
By Proposition \ref{L^1(V) is QEL and Divisibility Theorem}$(i)$, $V^{(1)}\subseteq\mathbb{P}^{24}$ is a QEL-manifold of dimension $15$ with secant defect $7$ and $\text{Sec}(V^{(1)})=\mathbb{P}^{24}$. Moreover, $V^{(1)}$ is a quadratic manifold (see Subsection \ref{subsection: quadratic manifolds}), and a SQEL-manifold by Lemma \ref{V^1 is a SQEL-mfd lem.}. However, by Proposition \ref{not exist 15-dim. quadratic SQEL-mfd s.t. Sec=P^24}, such a quadratic SQEL-manifold as $V^{(1)}$ does not exist. The conclusion follows.
\end{proof}

\subsection{\normalsize Proof of Proposition \ref{not exist 15-dim. quadratic SQEL-mfd s.t. Sec=P^24}} \label{subsection: Proof of non-existence of certain 15-dim. quadratic QEL-manifolds}

Let $V\subseteq\mathbb{P}^{r}$ be a smooth projective variety. Take a general point $v\in V$. Denote by $\pi_{v}: V\dashrightarrow V'$ the restriction to $V$ of the linear projection from $\mathbb{T}_{v}(V)$. We call the rational map $\pi_{v}$ the \textit{tangential projection} at $v$. Let $\pi: V\dashrightarrow V'$ be a rational map. Take a point $v'\in V'$. For the convenience of discussion, we use $\pi^{-1}(v')$ to denote the closure of the fiber of $v'$.

To prove Proposition \ref{not exist 15-dim. quadratic SQEL-mfd s.t. Sec=P^24}, we need to recall some properties on tangential projections and birational maps.

\begin{prop} (\cite[Thm. 2.3]{IR08}) \label{fiber of tangential projection}
Let $V\subseteq\mathbb{P}^{r}$ be a QEL-manifold with secant defect $\delta(V)>0$. Take a general point $v\in V$. Denote by $\pi_{v}: V\dashrightarrow V'$ the tangential projection at $v$. Then for a general point $y\in V$, $\pi_{v}^{-1}(\pi_{v}(y))$ is the entry locus of a general point $p\in\langle v, y\rangle$, i.e. a smooth quadric hypersurface.
\end{prop}

\begin{prop} (\cite[Prop. 1.3]{ES89}) \label{ein-shepherd pic =Z+Z then E irreducible}
Let $f: V'\dashrightarrow V$ be a proper birational map between two smooth projective varieties. Let $M=\{v\in V|\dim(f^{-1}(v))\geq 1\}$ and $E=f^{-1}(M)$. Suppose that $\text{Pic}(V')=\mathbb{Z}\oplus\mathbb{Z}$. Then $E$ is irreducible.
\end{prop}

\begin{prop} (\cite[Thm. 1.1]{ES89}) \label{ein-shepherd blowing up}
Let $f: V'\rightarrow V$ be a proper birational morphism between two smooth varieties. Let $M=\{v\in V|\dim(f^{-1}(v))\geq 1\}$, $E=f^{-1}(M)$ and $E_{1}=(E)_{\text{red}}$. Assume that $M$ is smooth and $E_{1}$ is an irreducible divisor. Then $E=E_{1}$ and $V'$ is the blow up of $V$ along $M$.
\end{prop}

Now we are ready to prove Proposition \ref{not exist 15-dim. quadratic SQEL-mfd s.t. Sec=P^24}.

\begin{proof}[Proof of Proposition \ref{not exist 15-dim. quadratic SQEL-mfd s.t. Sec=P^24}]
Assume such a projective manifold $V$ exists. Then the secant defect $\delta(V)=7$. Take a general point $u\in\mathbb{P}^{24}$. We can assume $u\in U_{g}(V)$. Then the entry locus $\Sigma_{u}$ is a $7$-dimensional irreducible smooth quadric hypersurface, and the secant cone $C_{u}$ of the entry locus is an $8$-dimensional linear subspace. Consider the linear projection $\mathbb{P}^{24}\dashrightarrow\mathbb{P}^{15}$ from $C_{u}$. Denote by $\pi: V\dasharrow\widetilde{V}$ the restriction to $V$ of the linear projection, where $\widetilde{V}=\pi(V)$. Remark that $C_{u}\cap V=\Sigma_{u}$ by Lemma \ref{C_u cap V=Sigma_u, C_(u_1) cap C_(u_2)=Sigma_(u_1) cap Sigma_(u_2) lem.}.

We claim that for any $y\in V\backslash\Sigma_{u}$, $\pi^{-1}\pi(y)$ is a linear space and $\pi^{-1}\pi(y)\cap\Sigma_{u}$ is a hyperplane of $\pi^{-1}\pi(y)$. The proof of this claim is the same with the discussion in \cite[Prop. 3.15]{AS13b}. Assume that $y_{1}$ and $y_{2}$ are two distinct points in $V\backslash\Sigma_{u}$ such that $\pi(y_{1})=\pi(y_{2})$. Denote by $u'=\langle y_{1}, y_{2}\rangle\cap C_{u}$. Then $u'\in\Sigma_{u}$, since otherwise $\{y_{1}, y_{2}\}\subseteq\Sigma_{u}=\Sigma_{u'}$ implying a contradiction, where the equality follows from the fact $V$ is a SQEL-manifold. Remark that $V$ is a quadratic manifold, and the line $\langle y_{1}, y_{2}\rangle$ intersects with $V$ at three distinct points $y_{1}$, $y_{2}$ and $u'$. Thus, $\langle y_{1}, y_{2}\rangle\subseteq V$. So the claim holds.

Denote by
$$M=\overline{\{(v_{1}, v_{2})\in\Sigma_{u}\times (V\backslash\Sigma_{u})|\langle v_{1}, v_{2}\rangle\subseteq V\}}\subseteq\Sigma_{u}\times V.$$

Let $p_{1}$ and $p_{2}$ be the two projections from $M$ to $\Sigma_{u}$ and $V$ respectively. Denote by $V_{u}=p_{2}(M)$. Then by the claim above,
$$V_{u}=\overline{\{v\in V\backslash\Sigma_{u}|\dim(\pi^{-1}\pi(v))\geq 1\}}.$$

Take a general point $v\in\Sigma_{u}$. Then $v$ is general in $V$. By Proposition \ref{L^1(V) is QEL and Divisibility Theorem}$(i)$, $\mathcal{L}_{v}(V)$ is an irreducible variety of dimension $9$, which implies that $p_{1}^{-1}(v)$ is an irreducible variety of dimension $10$. Hence, there is a unique irreducible component $\widetilde{M}$ of $M$ dominating $\Sigma_{u}$ by $p_{1}$. Moreover, $\dim(\widetilde{M})=17$. Denote by $\widetilde{V}_{u}=p_{2}(\widetilde{M})$ and $\widetilde{p}_{2}=p_{2}|_{\widetilde{M}}: \widetilde{M}\rightarrow\widetilde{V}_{u}$.
Denote by $c$ the codimension of $\widetilde{V}_{u}$ in $V$. Since $\dim(\widetilde{M})>2\dim(\Sigma_{u})$, we get that $\widetilde{V}_{u}\backslash\Sigma_{u}\neq\emptyset$. Take an arbitrary point $y\in\widetilde{V}_{u}\backslash\Sigma_{u}$. Then $\dim(p_{2}^{-1}(y))\geq\dim(\widetilde{M})-\dim(\widetilde{V}_{u})=c+2$. Moreover, $p_{2}^{-1}(y)=(\pi^{-1}\pi(y)\cap \Sigma_{u})\times\{y\}$, and $\pi^{-1}\pi(y)\cap \Sigma_{u}$ is a hyperplane in the linear space $\pi^{-1}\pi(y)$. Since $\Sigma_{u}$ is a smooth quadric hypersurface, we get that $\dim(p_{2}^{-1}(y))\leq\frac{\dim(\Sigma_{u})}{2}$. So $c=0$ or $1$. In particular, $\dim(V_{u})\geq\dim(\widetilde{V}_{u})\geq 14$.

\medskip

Case 1: Assume $\dim(V_{u})=15$. Then $V_{u}=V$ is irreducible. Remark that for any point $y\in V\backslash\Sigma_{u}$, $p_{2}^{-1}(y)$ is irreducible. Thus, $M\cap(\Sigma_{u}\times (V\backslash\Sigma_{u}))$ is irreducible. So $M$ is irreducible. In particular, $\widetilde{M}=M$ and $\widetilde{V}_{u}=V_{u}=V$. Take a general point $v\in\Sigma_{u}$. Now we consider the linear projection $\mathbb{P}^{24}\dashrightarrow\mathbb{P}^{16}$ from the tangent space $\mathbb{T}_{v}\Sigma_{u}$. Denote by $\pi_{u}: V\dashrightarrow\overline{V}$ the restriction to $V$ of this projection, where $\overline{V}=\pi_{u}(V)$. Then $\pi=\pi_{p}\circ\pi_{u}$, where $p=\pi_{u}(C_{u})$ is a point in $\overline{V}$, and $\pi_{p}: \overline{V}\dashrightarrow\pi_{p}(\overline{V})$ is the restriction to $\overline{V}$ of the projection $\mathbb{P}^{16}\dashrightarrow\mathbb{P}^{15}$ from $p$. We have the following commutative diagram:
\begin{eqnarray*}
\xymatrix{&V\ar@{-->}[ld]_-{\pi}\ar@{-->}[d]^-{\pi_{u}}\ar@{-->}[rd]^-{\pi_{v}}&\\
\widetilde{V}&\overline{V}\ar@{-->}[l]_-{\pi_{p}}\ar@{-->}[r]&\pi_{v}(V),
}
\end{eqnarray*}
where $\pi_{v}: V\dashrightarrow\pi_{v}(V)$ the tangential projection at $v$. Remark that $\pi$, $\pi_{u}$ and $\pi_{v}$ are restrictions to $V$ of linear projections from $\mathbb{P}^{24}$ with the center being $C_{u}$, $\mathbb{T}_{v}\Sigma_{u}$ and $\mathbb{T}_{v}V$, respectively.

Take a general point $v'\in V$. Since $\pi^{-1}\pi(v')$ is a linear space and $\pi^{-1}\pi(v')\cap\Sigma_{u}=\pi^{-1}\pi(v')\cap C_{u}$ is a hyperplane of $\pi^{-1}\pi(v')$, we get that $\pi_{u}(\pi^{-1}\pi(v'))$ is a line passing through $p\in\overline{V}$. Hence, $\pi_{u}^{-1}\pi_{u}(v')$ is a hyperplane of $\pi^{-1}\pi(v')$. Thus,
$$\dim(\pi_{u}^{-1}\pi_{u}(v'))=\dim(\pi^{-1}\pi(v'))-1=\dim(\widetilde{M})-\dim(\widetilde{V}_{u})=2.$$

By Proposition \ref{fiber of tangential projection},  the fiber $\pi_{v}^{-1}\pi_{v}(v')$ is the entry locus $\Sigma_{u'}$ passing through $v$ and $v'$, where $u'\in\langle v, v'\rangle$ and $u'\notin V$. By the generality of the choice of $u$, $v$ and $v'$, we get that $u'\in U_{g}(V)$, and  $\Sigma_{u}\neq\Sigma_{u'}$. Note that $u\in U_{g}(V)$.
Then by Lemma \ref{C_u cap V=Sigma_u, C_(u_1) cap C_(u_2)=Sigma_(u_1) cap Sigma_(u_2) lem.}, $\Sigma_{u}\cap\Sigma_{u'}=C_{u}\cap C_{u'}$ is a linear subspace. Let $\mathbb{P}_{uu'}=C_{u}\cap C_{u'}$, $\mathbb{P}_{uv'}=\langle v', \mathbb{P}_{uu'}\rangle$, and $s=\dim(\mathbb{P}_{uu'})$.  If $\mathbb{P}_{uv'}\subseteq\Sigma_{u'}$, then $\langle v, v'\rangle \subseteq\mathbb{P}_{uv'}\subseteq\Sigma_{u'}\subseteq V$. By the generality of $v$ and $v'$, $V$ is a linear subspace. This contradicts the assumption $\text{Sec}(V)=\mathbb{P}^{24}$. Hence, $\mathbb{P}_{uv'}\nsubseteq\Sigma_{u'}$. Remark that $\Sigma_{u'}\subseteq C_{u'}$ is a quadric hypersurface and $\mathbb{P}_{uv'}\subseteq C_{u'}$. Then there is an $s$-dimensional linear subspace $\widetilde{\mathbb{P}}_{uu'}$ containing $v'$ such that $\mathbb{P}_{uv'}\cap\Sigma_{u'}=\mathbb{P}_{uu'}\cup\widetilde{\mathbb{P}}_{uu'}$. Hence,
$$\pi_{u}^{-1}\pi_{u}(v')=\overline{(\langle\mathbb{T}_{v}\Sigma_{u}, v'\rangle\cap\Sigma_{u'})\backslash\Sigma_{u}}=\widetilde{\mathbb{P}}_{uu'}.$$
Thus, $s=\dim(\widetilde{\mathbb{P}}_{uu'})=\dim(\pi_{u}^{-1}\pi_{u}(v'))=2$.

By Proposition \ref{L^1(V) is QEL and Divisibility Theorem}$(i)$, $V^{(1)}\subseteq\mathbb{P}^{14}$ is a QEL-manifold of dimension $9$ with secant defect $5$. Then by Proposition \ref{classification of QEL with delta>n/2}, $V^{(1)}\subseteq\mathbb{P}^{14}$ is projectively equivalent to a nonsingular section of $S^{10}$ by a hyperplane $L$ in $\mathbb{P}^{15}$. Note that the intersection $\mathcal{L}_{v}(\Sigma_{u})\cap\mathcal{L}_{v}(\Sigma_{u'})=\mathcal{L}_{v}(\mathbb{P}_{uu'})$ is a projective line. Since $\mathcal{L}_{v}(\Sigma_{u})$ and $\mathcal{L}_{v}(\Sigma_{u'})$ are two different entry loci of $V^{(1)}$ with a nonempty intersection, we know from Example \ref{entry loci on S^10 e.g.} that $\dim(\mathcal{L}_{v}(\Sigma_{u})\cap\mathcal{L}_{v}(\Sigma_{u'}))\geq 2$, which is a contradiction.

\medskip

Case 2: Assume $\dim(V_{u})=14$. Then $\dim(\widetilde{V}_{u})=14$ and $c=1$. Since the closures of fibers of $\pi$ are linear, it is a birational map with exceptional locus $V_{u}$ (set-theoretically), and $\widetilde{V}=\pi(V)=\mathbb{P}^{15}$.

Denote by $V'=\text{Bl}_{\Sigma_{u}}V$, and $f: V'\rightarrow V$ the blow-up morphism. Since $V$ is a quadratic manifold, $\Sigma_{u}$ is the scheme-theoretic intersection of $C_{u}$ and $V$. Hence, $\Sigma_{u}$ is the base locus scheme of $\pi$. Hence, there is a morphism $g: V'\rightarrow \mathbb{P}^{15}$ such that $\pi=g\circ f^{-1}$. Equip $V_{u}$ with the reduced closed subscheme structure. Denote by $V'_{u}=\text{Bl}_{\Sigma_{u}}V_{u}$ and $B=g(V'_{u})$. So we have the following commutative diagram:
\begin{eqnarray*}
\xymatrix{&&V'_{u}=\text{Bl}_{\Sigma_{u}}V_{u}\ar[2, -2]_-{f|_{V'_{u}}}\ar[d]\ar[2, 2]^-{g|_{V'_{u}}}&&\\
&&V'=\text{Bl}_{\Sigma_{u}}V\ar[ld]_-{f}\ar[rd]^-{g}&&\\
V_{u}\ar[r]&V\ar@{-->}[0, 2]^-{\pi}&&\widetilde{V}=\mathbb{P}^{15}&B=g(V'_{u}),\ar[l]
}
\end{eqnarray*}
where the morphisms $V_{u}\rightarrow V$, $V'_{u}\rightarrow V'$ and $B\rightarrow\widetilde{V}$ are natural inclusions.

Let $E_{V}=f^{-1}(\Sigma_{u})$, and $H_{V}$ (resp. $H$) be the pull-back of a hyperplane section of $V$ (resp. of $\mathbb{P}^{15}$). Denote by $K_{V'}$ the canonical divisor of $V'$. Since $V'=\text{Bl}_{\Sigma_{u}}V$, we get that
\begin{eqnarray}
&&\text{Pic}(V')=\mathbb{Z}(H_{V})\oplus\mathbb{Z}(E_{V}); \label{eqn. Pic(V')=Z+Z}\\
&&-K_{V'}=11H_{V}-7E_{V};  \label{eqn. -K_V'=11H_V-7E_V}\\
&&H=H_{V}-E_{V}.  \label{eqn. H=H_V-E_V}
\end{eqnarray}

Let $B_{1}=\{y\in\mathbb{P}^{15}\mid \dim(g^{-1}(y))\geq 1\}$. For any point $v\in V'_{u}\backslash E_{V}$,
$$\dim(g^{-1}g(v))=\dim(f^{-1}\pi^{-1}\pi f(v))=\dim(\pi^{-1}\pi f(v))\geq 1,$$
where the second equality follows from the fact that $\pi^{-1}\pi f(v)$ is a linear subspace (hence an irreducible variety) not contained in $\Sigma_{u}$. In particular, $g(V'_{u}\backslash E_{V})\subseteq B_{1}$. Hence, $B=g(V'_{u})\subseteq B_{1}$ and $V'_{u}\subseteq (g^{-1}(B))_{\text{red}}\subseteq (g^{-1}(B_{1}))_{\text{red}}$. By Proposition \ref{ein-shepherd pic =Z+Z then E irreducible}, $g^{-1}(B_{1})$ is irreducible. Since $\dim(V'_{u})=\dim(V')-1$, we get that $V'_{u}=(g^{-1}(B))_{\text{red}}=(g^{-1}(B_{1}))_{\text{red}}$, and $V'_{u}$ is an irreducible divisor. Thus, $\widetilde{V}_{u}=V_{u}=f(V'_{u})$, $B=g(V'_{u})=B_{1}$, and $V_{u}$ and $B$ are irreducible.

Remark that $B=\pi(V_{u})=\pi(\widetilde{V}_{u})$ and $V_{u}=\pi^{-1}(B)$. For a general point $b\in B$, $\pi^{-1}(b)$ is a linear subspace such that
$$\dim(\pi^{-1}(b))=\dim(\widetilde{M})-\dim(\widetilde{V}_{u})+1=4,$$
and $\pi^{-1}(b)\cap C_{u}$ is a linear subspace of dimension $3$ contained in $\Sigma_{u}$. Thus,
$$\dim(B)=\dim(V_{u})-\dim(\pi^{-1}(b))=10.$$

Denote by $B^{o}=B\backslash\text{Sing}(B)$, and $U=\mathbb{P}^{15}\backslash\text{Sing}(B)$, where $\text{Sing}(B)$ is the singular locus of $B$. By Proposition \ref{ein-shepherd blowing up}, $g^{-1}(B^{o})=g^{-1}(B^{o})_{\text{red}}= V'_{u}\cap g^{-1}(U)$, and $g^{-1}(U)\rightarrow U$ is the blow up of $U$ along $B_{1}^{o}$. Hence, the canonical divisor $K_{g^{-1}(U)}=g^{*}(K_{U})+4V'_{u}|_{g^{-1}(U)}$. So $-K_{V'}=16H-4V'_{u}$. Combining with the formulae (\ref{eqn. -K_V'=11H_V-7E_V})(\ref{eqn. H=H_V-E_V}), we get $V'_{u}=\frac{5}{4}H_{V}-\frac{9}{4}E_{V}$. Remark that $\text{Pic}(V')=\mathbb{Z}(H_{V})\oplus\mathbb{Z}(E_{V})$. Then $V'_{u}$ is not a Cartier divisor, which contradicts the smoothness of $V'$. This finishes the proof.
\end{proof}

\small

Qifeng Li

\smallskip

Institute of Mathematics, AMSS, Chinese Academy of Sciences,

\smallskip

55 Zhongguancun East Road, Beijing, 100190, P. R. China

\smallskip

E-mail address: qifengli@amss.ac.cn

\end{document}